\documentclass[12pt, reqno]{amsart}
\usepackage{amsmath, amsthm, amscd, amsfonts, amssymb, graphicx, color}
\usepackage[bookmarksnumbered, colorlinks, plainpages]{hyperref}

\newtheorem{theorem}{Theorem}[section]
\newtheorem{lemma}[theorem]{Lemma}

\newtheorem{corollary}[theorem]{Corollary}
\theoremstyle{definition}

\theoremstyle{remark}
\newtheorem{remark}[theorem]{Remark}
\numberwithin{equation}{section}

\def\cF{\mathcal{F}}
\def\om{\mathrm{om}}
\def\oc{\mathrm{oc}}
\def\ffi{\varphi}
\def\bR{\mathbb{R}}

\def\supp{\mathrm{supp}\,}
\def\eps{\varepsilon}
\def\conv{\mathrm{conv}\,}
\def\ex{\mathrm{ex}\,}

\begin{document}
\setcounter{page}{1}

\title[The non-negative operator convex functions on $(0,\infty)$]{Conic structure of the non-negative operator convex functions on $(0,\infty)$}

\author[U.\ Franz, F.\ Hiai]{Uwe Franz$^1$ and Fumio Hiai$^2$$^{*}$}

\address{$^{1}$ D\'epartement de math\'ematiques de Besan\c con,
Universit\'e de Franche-Comt\'e, 16, route de Gray, 25 030 Besan\c con cedex, France}
\email{\textcolor[rgb]{0.00,0.00,0.84}{uwe.franz@univ-fcomte.fr}}

\address{$^{2}$ Tohoku University (Emeritus),
Hakusan 3-8-16-303, Abiko 270-1154, Japan}
\email{\textcolor[rgb]{0.00,0.00,0.84}{hiai.fumio@gmail.com}}

\dedicatory{Dedicated to Professor Tsuyoshi Ando with admiration}



\date{Received: xxxxxx; Revised: yyyyyy; Accepted: zzzzzz.
\newline \indent $^{*}$ Corresponding author}

\begin{abstract}
The conic structure of the convex cone of non-negative operator convex functions on $(0,\infty)$ (also on $(-1,1)$) is clarified. We completely determine the extreme rays, the closed faces, and the simplicial closed faces of this convex cone.
\end{abstract} \maketitle

\section{Introduction}

\noindent A real continuous function $f$ on an interval $J$ of the real line is said to be
{\it operator monotone} if $A\ge B$ implies $f(A)\ge f(B)$ for bounded self-adjoint
operators $A,B$ with spectra in $J$ on an infinite-dimensional (separable) Hilbert space,
where $f(A)$ denotes the usual functional calculus of $A$, and it is said to be
{\it operator monotone decreasing} if $-f$ is operator monotone. Also, the function $f$
is said to be {\it operator convex} if $f((1-t)A+tB)\le(1-t)f(A)+tf(B)$ for all such $A,B$
and $t\in(0,1)$. These operator monotonicity or convexity is equivalent to that $f$
satisfies the same property for every $n\times n$ Hermitian matrices $A,B$ with
eigenvalues in $J$ (i.e., $f$ is matrix monotone or convex of order $n$) for every
positive integer $n$. The theory of operator/matrix monotone functions was initiated by
the celebrated paper of L\"owner \cite{Lo}, which was soon followed by Kraus \cite{Kr} on
operator/matrix convex functions. A modern treatment of operator monotone and convex
functions was established by Hansen and Pedersen in their seminal paper \cite{HP}. In
\cite{Do,An,Bh} (also \cite{Hi}) are found comprehensive expositions on the subject matter.

In L\"owner's theory it is well-known (see e.g., \cite[Sect.\ V.4]{Bh},
\cite[Sect.\ 2.7]{Hi}) that a non-negative operator monotone function $f$ on the interval
$(0,\infty)$ admits the integral representation
$$
f(x)=\int_{[0,\infty]}{x(1+\lambda)\over x+\lambda}\,d\mu(\lambda),\qquad
x\in(0,\infty),
$$
where $\mu$ is a unique finite positive measure on $[0,\infty]$ and
$x(1+\lambda)/(x+\lambda)$ means $x$ for $\lambda=\infty$. In the paper we always consider
$[0,\infty]=[0,\infty)\cup\{\infty\}$ as the one-point compactification of $[0,\infty)$.
The above unique representation shows that the convex cone $\cF_\om^+(0,\infty)$ of
non-negative operator monotone functions on $(0,\infty)$ is order isomorphic to that of
finite positive measures on $[0,\infty]$ (identified with the positive part of the dual
space of the Banach space $C([0,\infty])$ of real continuous functions on $[0,\infty]$).
Hence $\cF_\om^+(0,\infty)$ is a simplicial convex cone whose base
$\{f\in\cF_\om^+(0,\infty):f(1)=1\}$ is a Bauer simplex. Thus, the conic structure of
$\cF_\om^+(0,\infty)$ is very simple. However, the convex cone $\cF_\oc^+(0,\infty)$ of
non-negative operator convex functions on the interval $(0,\infty)$ is not so simple as
$\cF_\om^+(0,\infty)$, whose conic structure is our main concern in the present paper.

In the previous paper \cite{FHR} we introduced the notion of {\it operator $k$-tone}
functions on an open interval, extending operator monotone and convex functions. The
integral representations for operator $k$-tone functions are useful here, so that we
state them in Section 2 of this paper in the form specialized to the case $k=2$ (the
case of operator convex functions). In Section 3 we determine all extreme rays of
$\cF_\oc^+(0,\infty)$. To do this, we consider facial subcones
$F_\alpha:=\{f\in\cF_\oc^+(0,\infty):f(\alpha)=0\}$ for $0\le\alpha<\infty$ and
$F_\infty:=\{f\in\cF_\oc^+(0,\infty):\mbox{$f$ is non-increasing}\}$. In Section 4 we
determine all closed faces and all simplicial closed faces of $\cF_\oc^+(0,\infty)$.
Finally in Section 5, similar discussions are given for the convex cone of non-negative
operator convex functions on $(-1,1)$ though the detailed proofs are omitted.

\section{Integral representations}

The notion of operator $k$-tone functions in \cite{FHR} extends operator monotone and
operator convex functions. In fact, operator $1$-tone and operator $2$-tone functions are
operator monotone and operator convex functions, respectively. Since we are concerned
with operator convex functions, we here state the results \cite[Theorems 4.1, 5.1]{FHR}
under specialization to operator convex functions (the case $k=2$). They will play an
essential role in this paper.

\begin{theorem}\label{T-2.1}
A real function $f$ on $(-1,1)$ is operator convex if and only if it is $C^1$ on $(-1,1)$
and there exists a finite positive measure $\mu$ on $[-1,1]$ such that, for any choice of
$\alpha\in(-1,1)$,
$$
f(x)=f(\alpha)+f'(\alpha)(x-\alpha)
+\int_{[-1,1]}{(x-\alpha)^2\over(1-\lambda x)(1-\lambda\alpha)}\,d\mu(\lambda),
\qquad x\in(-1,1).
$$
\end{theorem}

\begin{theorem}\label{T-2.2}
A real function $f$ on $(0,\infty)$ is operator convex if and only if it is $C^1$ on
$(0,\infty)$ and there exist a constant $\gamma\ge0$ and a positive measure $\mu$ on
$[0,\infty)$ such that
$$
\int_{[0,\infty)}{1\over(1+\lambda)^3}\,d\mu(\lambda)<+\infty
$$
and, for any choice of $\alpha\in(0,\infty)$,
\begin{align}
f(x)&=f(\alpha)+f'(\alpha)(x-\alpha)+\gamma(x-\alpha)^2 \nonumber\\
&\qquad\qquad\qquad
+\int_{[0,\infty)}{(x-\alpha)^2\over(x+\lambda)(\alpha+\lambda)^2}\,d\mu(\lambda),
\qquad x\in(0,\infty). \label{F-2.1}
\end{align}
\end{theorem}

A remarkable point in the above theorems is that the measure $\mu$ as well as the constant
$\gamma$ is independent of the choice of $\alpha$. Furthermore, note that $\mu$ as well as
$\gamma$ is uniquely determined even for a fixed $\alpha$. Indeed, if $f$ is operator
convex on $(0,\infty)$ and $\alpha\in(0,\infty)$, then $h(x):=(f(x)-f(\alpha))/(x-\alpha)$
is operator monotone on $(0,\infty)$ due to the theorem of Kraus \cite{Kr} (also
\cite[Corollary 2.4.6]{Hi}) and the integral expression of Theorem \ref{T-2.2} gives
$$
h(x)=h(\alpha)+\gamma(x-\alpha)
+\int_{[0,\infty)}{x-\alpha\over(x+\lambda)(\alpha+\lambda)}\,d\nu(\lambda),
\qquad x\in(0,\infty),
$$
where $d\nu(\lambda):=(\alpha+\lambda)^{-1}\,d\mu(\lambda)$. As seen in the proof of
\cite[Theorem 1.10]{FHR}, the measure $\nu$ (hence $\mu$) and $\gamma\ge0$ are unique in
this expression of $h$. The discussion is similar for $f$ in Theorem \ref{T-2.1} (see
the proof of \cite[Theorem 1.8]{FHR}). We call the measure $\mu$ in the above theorems
the {\it representing measure} of $f$. The $\mu$ and $\gamma$ are the canonical data of
an operator convex function $f$ on $(0,\infty)$ determining $f$ up to its linear term.

\section{Extreme rays}

In this and the next sections we shall clarify the structures of extreme rays and of
closed facial subcones in the convex cone of non-negative operator convex functions on
$(0,\infty)$. We write $\cF(0,\infty)$ for the locally convex topological space consisting
of all real functions on $(0,\infty)$ with the pointwise convergence topology, and let
$\cF^+(0,\infty):=\{f\in\cF(0,\infty):f\ge0\}$. Let $\cF_\oc^+(0,\infty)$ denote the set
of all operator convex functions $f\ge0$ on $(0,\infty)$ and $\cF_\oc^{++}(0,\infty)$ the
set of all operator convex functions $f>0$ (strictly positive) on $(0,\infty)$, which are
convex cones in $\cF(0,\infty)$. Since operator convexity is preserved under pointwise
convergence, note that $\cF_\oc^+(0,\infty)$ is closed in $\cF(0,\infty)$ while
$\cF_\oc^{++}(0,\infty)$ is not. As easily seen, note also that the pointwise convergence
topology on $\cF_\oc^+(0,\infty)$ is metrizable.

When $C$ is a convex set of a real vector space, a convex subset $F\subset C$ is called
a {\it face} of $C$ if, for any $x,y\in C$, $\lambda x+(1-\lambda)y\in F$ for some
$\lambda\in(0,1)$ implies that $x,y\in F$. When $C$ is a convex cone, recall that a
{\it ray} in $C$ is a set of the form $R=\{\lambda x:0<\lambda<\infty\}$ for some
$x\in C$, $x\not=0$. In this case we say that $R$ is given by $x$. The ray $R$ is called
an {\it extreme ray} if it is a face of $C$, or equivalently, for $y,z\in C$,
$y+z\in R$ implies $y,z\in R$. We fix a few more terminologies in our special situation.
In this paper we simply say that $F$ is a {\it face} of $\cF_\oc^+(0,\infty)$ to mean that
$F$ is a convex subcone of $\cF_\oc^+(0,\infty)$ which is a face of $\cF_\oc^+(0,\infty)$.
Such a face $F$ is always assumed to be non-trivial, i.e., $F\ne\{0\}$. We say also that
a convex subcone $F$ of $\cF_\oc^+(0,\infty)$ is {\it simplicial} if there is a compact
base $F^0$ of $F$ that is a ({\it Choquet}\,) {\it simplex}. Here, a subset $F^0\subset F$
is a {\it base} of $F$ if $F^0$ is convex, $0\not\in F^0$, and for every
$f\in F\setminus\{0\}$ there is a unique $c>0$ such that $c^{-1}f\in F^0$. For definition
of a simplex, see \cite{Ph,simon2011}. A simplex whose extreme points form a compact
subset is called a {\it Bauer simplex}.

The first two theorems of this section are concerned with the extreme rays of
$\cF_\oc^{++}(0,\infty)$ and of $\cF_\oc^+(0,\infty)$, respectively.

\begin{theorem}\label{T-3.1}
The extreme rays of $\cF_\oc^{++}(0,\infty)$ are given by one of the following elements:
\begin{equation}\label{F-3.1}
\begin{cases}
{1\over x+\lambda} & \text{$(0\le\lambda<\infty)$}, \\
\,1, \\
{x^2\over x+\lambda} & \text{$(0\le\lambda<\infty)$}, \\
\,x^2.
\end{cases}
\end{equation}
\end{theorem}

\begin{theorem}\label{T-3.2}
The extreme rays of $\cF_\oc^+(0,\infty)$ are given by one of the following elements:
\begin{equation}\label{F-3.2}
\begin{cases}
{1\over x+\lambda} & \text{$(0\le\lambda<\infty)$}, \\
\,1, \\
{(x-\alpha)^2\over x+\lambda} & \text{$(0\le\lambda<\infty,\ 0\le\alpha<\infty)$}, \\
(x-\alpha)^2 & \text{$(0\le\alpha<\infty)$}.
\end{cases}
\end{equation}
Furthermore, $\cF_\oc^+(0,\infty)$ is the closed convex cone generated by the above extreme
elements while it is not simplicial.
\end{theorem}

The third theorem of this section gives some typical simplicial closed faces of
$\cF_\oc^+(0,\infty)$ while the complete list of those faces will be presented in the next
section.

\begin{theorem}\label{T-3.3}
The convex cone $\cF_\oc^+(0,\infty)$ has the following faces.
\begin{itemize}
\item[(i)] The closed convex subcone generated by $1/(x+\lambda)$ $(0\le\lambda<\infty)$
and $1$ is
$$
F_\infty:=\{f\in\cF_\oc^+(0,\infty):\mbox{$f$ is non-increasing}\},
$$
which becomes a face of $\cF_\oc^+(0,\infty)$.
\item[(ii)] For each fixed $\alpha\in[0,\infty)$ the closed convex subcone generated
by\break $(x-\alpha)^2/(x+\lambda)$ $(0\le\lambda<\infty)$ and $(x-\alpha)^2$ is
\begin{equation}\label{F-3.3}
\qquad F_\alpha:=\{f\in\cF_\oc^+(0,\infty):f(\alpha)=0\}
\quad(\mbox{where $f(\alpha)=f(+0)$ for $\alpha=0$}),
\end{equation}
which becomes a face of $\cF_\oc^+(0,\infty)$.
\item[(iii)] The above closed faces $F_\alpha$ $(0\le\alpha\le\infty)$ are simplicial.
Indeed, each of them is generated by a Bauer simplex; more precisely, each of
\begin{align}
F_\infty^0&:=\{f\in F_\infty:f(1)=1\}, \label{F-3.4}\\
F_\alpha^0&:=\{f\in F_\alpha:f(\alpha+1)=1\}\quad(0\le\alpha<\infty) \label{F-3.5}
\end{align}
is a compact convex subset of $\cF(0,\infty)$ and is a Bauer simplex. (In the above,
$F_\alpha^0$ can also be taken by $\{f\in F_\alpha:f(\xi)=1\}$ for any
$\xi\in(0,\infty)\setminus\{\alpha\}$.)
\end{itemize}
\end{theorem}

To prove the theorems, it is convenient to prove Theorem \ref{T-3.3} first.

\medskip
\noindent{\it Proof of Theorem \ref{T-3.3}.}\enspace
(i)\enspace
It is known \cite[Theorem 3.1]{AH} that a function $f\in\cF^+(0,\infty)$ is operator
monotone decreasing if and only if it is operator convex and non-increasing in the
numerical sense. Therefore, one can write
\begin{equation}\label{F-3.6}
F_\infty=\{f\in\cF^+(0,\infty):\mbox{$f$ is operator monotone decreasing}\}.
\end{equation}
By \cite{Ha} (also \cite[Theorem 3.1]{AH}), $f\in\cF^+(0,\infty)$ belongs to $F_\infty$
if and only if it admits the integral representation
\begin{equation}\label{F-3.7}
f(x)=a+\int_{[0,\infty)}{1+\lambda\over x+\lambda}\,d\nu(\lambda),
\end{equation}
where $a\ge0$ and $\nu$ is a finite positive measure on $[0,\infty)$. This
representation shows that $F_\infty$ is the closed convex cone generated by
$1/(x+\lambda)$ ($0\le\lambda<\infty$) and $1$. Here, the closedness of $F_\infty$
follows since operator monotone decreasingness as well as non-negativity is closed under
the pointwise convergence topology of $\cF(0,\infty)$.

Since $f(1)=0$ implies $f=0$ for any $f\in F_\infty$, we have
$F_\infty=\{cf:f\in F_\infty^0,\,c\ge0\}$, where $F_\infty^0$ is defined in \eqref{F-3.4}.
Note that $f$ in \eqref{F-3.7} belongs to $F_\infty^0$ if and only if
$$
a+\nu([0,\infty))=1.
$$
In this case, for every $x\in(0,1)$ we have
$$
f(x)\le a+{1\over x}\,\nu([0,\infty))\le1+{1\over x},
$$
and $f(x)\le f(1)=1$ for all $x\ge1$. Since $F_\infty^0$ is closed under pointwise
convergence, this implies by Tychonoff's theorem that $F_\infty^0$ is a compact convex
subset of $\cF(0,\infty)$. Moreover, since \eqref{F-3.7} yields the well-known integral
representation of the operator monotone function $f(x^{-1})$ that is unique (see e.g.,
\cite[Theorem 2.7.11]{Hi}), we see that the representation \eqref{F-3.7} is also unique
(more precisely, $a$ and $\nu$ are unique). Therefore, it follows that the extreme points
of $F_\infty^0$ are
$$
\ffi_\lambda(x):={1+\lambda\over x+\lambda}\quad(\lambda\in[0,\infty)),\qquad
\ffi_\infty(x):=1,
$$
so that \eqref{F-3.7} is a unique integral decomposition of $f$ into the extreme points.
Since $F_\infty^0$ is metrizable, it follows from the Choquet theory (see
\cite[Sect.\ 10]{Ph}, \cite[Chap.\ 11]{simon2011}) that $F_\infty^0$ is a simplex.
Moreover, $\{\ffi_\lambda:\lambda\in[0,\infty]\}$ in $\cF(0,\infty)$ is homeomorphic to
the compact set $[0,\infty]$, the one-point compactification of $[0,\infty)$. Hence
$F_\infty^0$ is a Bauer simplex.

If a function $f$ is operator convex, then it is $C^1$ (indeed, even analytic) and
numerically convex. So $f'(x)$ exists and is increasing. Therefore, the limit
$f'(\infty):=\lim_{x\to\infty}f'(x)$ exists in $[0,+\infty]$ for any
$f\in\cF_\oc^+(0,\infty)$. One can then characterize $F_\infty$ as
\[
F_\infty=\{f\in \cF_\oc^+(0,\infty): f'(\infty) =0\}.
\]
This shows that $F_\infty$ is a face of $\cF_\oc^+(0,\infty)$, and so each $\ffi_\lambda$
($0\le\lambda\le\infty$) gives an extreme ray of $\cF_\oc^+(0,\infty)$.

\medskip
(ii)\enspace{\it Case of $\alpha=0$.}\enspace
It is convenient to divide the proof of (ii) into two cases $\alpha=0$ and
$\alpha\in(0,\infty)$. It is obvious that $F_0$ in \eqref{F-3.3} is a face of
$\cF_\oc^+(0,\infty)$. Note (see \cite[Theorem V.2.9]{Bh}) that $f\in\cF^+(0,\infty)$
with $f(+0)=0$ is operator convex on $(0,\infty)$ if and only if $f(x)/x$ is operator
monotone on $(0,\infty)$. Hence $f\in\cF^+(0,\infty)$ belongs to $F_0$ if and only if
$f(x)/x$ admits the integral representation (see \cite[(V.53)]{Bh} or
\cite[Theorem 2.7.11]{Hi})
\begin{equation}\label{F-3.8}
{f(x)\over x}=a+b x+\int_{(0,\infty)}{x(1+\lambda)\over x+\lambda}\,d\mu(\lambda),
\end{equation}
where $a=\lim_{x\searrow0}f(x)/x$, $b=\lim_{x\to\infty}f(x)/x^2$, and $\mu$ is a finite
positive measure on $(0,\infty)$. By extending $\mu$ to a measure on $[0,\infty)$ with
$\mu(\{0\})=a$, \eqref{F-3.8} is rewritten as
\begin{equation}\label{F-3.9}
f(x)=bx^2+\int_{[0,\infty)}{x^2(1+\lambda)\over x+\lambda}\,d\mu(\lambda).
\end{equation}
This shows (ii) in the case $\alpha=0$. Here, the closedness of $F_0$ is seen as follows.
Let $\{f_n\}$ be a sequence in $F_0$ converging to $f\in\cF_\oc^+(0,\infty)$. Choose an
$a>f(1)$; then for every $n$ sufficiently large, $f_n(1)<a$ and so $f_n(x)<ax$ for all
$x\in(0,1)$ by convexity of $f_n$. Hence $f(x)\le ax$ for all $x\in(0,1)$ and $f(+0)=0$
follows.

Since $f(1)=0$ implies $f=0$ for any $f\in F_0$, we have $F_0=\{cf:f\in F_0^0,\,c\ge0\}$,
where $F_0^0$ is defined in \eqref{F-3.5}. Note that $f$ in \eqref{F-3.9} belongs to
$F_0^0$ if and only if
$$
b+\mu([0,\infty))=1.
$$
In this case, for every $x\ge1$ we have
$$
f(x)\le bx^2+x^2\mu([0,\infty))\le2x^2,
$$
and $f(x)\le f(1)=1$ for all $x\in(0,1)$. This implies that $F_0^0$ is a compact convex
subset of $\cF(0,\infty)$. Since the representation \eqref{F-3.9} is unique as so is
\eqref{F-3.8}, it follows that $F_0^0$ is a simplex whose extreme points are
$$
\psi_\lambda(x):={x^2(1+\lambda)\over x+\lambda}\quad(0\le\lambda<\infty),\qquad
\psi_\infty(x):=x^2.
$$
Since $F_0$ is a face of $\cF_\oc^+(0,\infty)$ as mentioned above, it also follows that
each $\psi_\lambda$ ($0\le\lambda\le\infty$) gives an extreme ray of
$\cF_\oc^+(0,\infty)$. Moreover, $\{\psi_\lambda:\lambda\in[0,\infty]\}$ in $\cF(0,\infty)$
is homeomorphic to $[0,\infty]$ so that $F_0^0$ is a Bauer simplex.

\medskip
(ii)\enspace{\it Case of $\alpha\in(0,\infty)$.}\enspace
Let $\alpha$ be arbitrarily fixed in $(0,\infty)$. It is obvious that $F_\alpha$ in
\eqref{F-3.3} is a face of $\cF_\oc^+(0,\infty)$. For each $f\in F_\alpha$, since
$f(\alpha)=f'(\alpha)=0$, the integral representation \eqref{F-2.1} is rewritten as
\begin{equation}\label{F-3.10}
f(x)=\gamma(x-\alpha)^2
+\int_{[0,\infty)}{(x-\alpha)^2(1+\alpha+\lambda)\over x+\lambda}\,d\nu(\lambda),
\end{equation}
where $\gamma\ge0$ and $\nu$ is a finite measure on $[0,\infty)$ defined by
$$
d\nu(\lambda):={1\over(\alpha+\lambda)^2(1+\alpha+\lambda)}\,d\mu(\lambda)
$$
from $\mu$ in Theorem \ref{T-2.2}. Conversely, since $(x-\alpha)^2$ and
$(x-\alpha)^2/(x+\lambda)$ ($0\le\lambda<\infty$) are in $F_\alpha$, $f$ belongs to
$F_\alpha$ whenever it is of the form \eqref{F-3.10}. Hence $f\in\cF^+(0,\infty)$ belongs
to $F_\alpha$ if and only if it admits the integral representation \eqref{F-3.10}. This
shows (ii) in the case $\alpha\in(0,\infty)$.

Since $f(1+\alpha)=0$ implies $f=0$ for any $f\in F_\alpha$, we have
$F_\alpha=\{cf:f\in F_\alpha^0,\,c\ge0\}$, where $F_\alpha^0$ is as in \eqref{F-3.5}.
Note that $f$ in \eqref{F-3.10} belongs $F_\alpha^0$ if and only if
$$
\gamma+\nu([0,\infty))=1.
$$
In this case, for every $x\ge\alpha$ we have
$$
f(x)\le\gamma(x-\alpha)^2+{(x-\alpha)^2(1+\alpha)\over\alpha}\,\nu([0,\infty))
\le{(x-\alpha)^2(1+2\alpha)\over\alpha},
$$
and for every $x\in(0,\alpha)$,
$$
f(x)\le\gamma(x-\alpha)^2+{(x-\alpha)^2(1+\alpha)\over x}\,\nu([0,\infty))
\le{(x-\alpha)^2(x+1+\alpha)\over x}.
$$
This implies that $F_\alpha^0$ is a compact convex subset of $\cF(0,\infty)$. Moreover,
since the representation \eqref{F-2.1} is unique and \eqref{F-3.10} uniquely corresponds
to \eqref{F-2.1}, the representation \eqref{F-3.10} is also unique. Hence $F_\alpha^0$ is
a simplex whose extreme points are
$$
\psi_{\alpha,\lambda}(x):={(x-\alpha)^2(1+\alpha+\lambda)\over x+\lambda}\quad
(0\le\lambda<\infty),\qquad\psi_{\alpha,\infty}(x):=(x-\alpha)^2.
$$
Since $F_\alpha$ is a face of $\cF_\oc^+(0,\infty)$, it follows that each
$\psi_{\alpha,\lambda}$ ($0\le\lambda\le\infty$) gives an extreme ray of
$\cF_\oc^+(0,\infty)$. Since $\{\psi_{\alpha,\lambda}:\lambda\in[0,\infty]\}$ in
$\cF(0,\infty)$ is homeomorphic to $[0,\infty]$, $F_\alpha^0$ is a Bauer simplex.

\medskip
(iii) has been proved in the course of the above proofs of (i) and (ii), and the last
comment in the parentheses is easily verified. For this purpose it might be noted that for
each $\xi\in(0,\infty)\setminus\{\alpha\}$ the set $\{f\in F_\alpha: f(\xi)=1\}$ is a base
of $F_\alpha$, which is a simplex whose extreme points are
\[
\frac{(x-\alpha)^2(\xi+\lambda)}{(\xi-\alpha)^2(x+\lambda)}\quad(0\le \lambda<\infty),
\qquad\frac{(x-\alpha)^2}{(\xi-\alpha)^2}.
\]
\qed

\medskip\noindent
{\it Proof of Theorem \ref{T-3.2}.}\enspace
It has been proved in the course of the proof of Theorem \ref{T-3.3} that the elements in
\eqref{F-3.2} give extreme rays of $\cF_\oc^+(0,\infty)$. If $f\in\cF_\oc^+(0,\infty)$
belongs to one of $F_\alpha$ ($0\le\alpha\le\infty$), then it has also been proved (see
\eqref{F-3.7}, \eqref{F-3.9} and \eqref{F-3.10}) that $f$ is written as an integral of
elements in those extreme rays by a finite positive measure. If
$f\in\cF_\oc^+(0,\infty)$ belongs to none of $F_\alpha$ ($0\le\alpha\le\infty$), then
there exists an $\alpha_0\in[0,\infty)$ such that $f(x)>f(\alpha_0)>0$ for all
$x\in(0,\infty)\setminus\{\alpha_0\}$. Now let $a_0:=f(\alpha_0)$; then
$f-a_0\in F_{\alpha_0}$. Hence by the proof of Theorem \ref{T-3.3}\,(ii) (see \eqref{F-3.9}
and \eqref{F-3.10}), $f$ admits the representation
\begin{equation}\label{F-3.11}
f(x)=a_0+b_0(x-\alpha_0)^2
+\int_{[0,\infty)}{(x-\alpha_0)^2(1+\alpha_0+\lambda)\over x+\lambda}\,d\nu_0(\lambda),
\end{equation}
where $b_0\ge0$ and $\nu_0$ is a finite positive measure on $[0,\infty)$. This shows
that $\cF_\oc^+(0,\infty)$ is the closed convex cone generated by the extreme elements in
\eqref{F-3.2}. Furthermore, the above integral representation implies that there are no
extreme elements other than those in \eqref{F-3.2}. Finally, to show that
$\cF_\oc^+(0,\infty)$ is not simplicial, it is enough to note that
\begin{equation}\label{F-3.12}
x+{1\over x+1}=1+{x^2\over x+1}
\end{equation}
is an element of $\cF_\oc^+(0,\infty)$ that is decomposed into extreme elements in two
different ways.\qed

\medskip\noindent
{\it Proof of Theorem \ref{T-3.1}.}\enspace
The elements in \eqref{F-3.1} are in $\cF_\oc^{++}(0,\infty)$ and give extreme rays of
$\cF_\oc^+(0,\infty)$ by Theorem \ref{T-3.2}. Hence they give extreme rays of
$\cF_\oc^{++}(0,\infty)$. Since they are extreme elements of either $F_0$ or $F_\infty$
that are faces of $\cF_\oc^+(0,\infty)$ (hence faces of
$\cF_\oc^{++}(0,\infty)\cup\{0\}$), it remains to show that any element
$f\in\cF_\oc^{++}(0,\infty)\setminus(F_0\cup F_\infty)$ is not extreme in
$\cF_\oc^{++}(0,\infty)$. For any such $f$, there exists an
$\alpha_0\in[0,\infty)$ such that $f(x)>f(\alpha_0)>0$ for all
$x\in(0,\infty)\setminus\{\alpha_0\}$. Let $a_0:=f(\alpha_0)/2$ and write $f=(f-a_0)+a_0$.
Since $f-a_0,\,a_0\in\cF_\oc^{++}(0,\infty)$ and $f-a_0\ne a_0$, $f$ is not extreme in
$\cF_\oc^{++}(0,\infty)$.\qed

\begin{remark}\label{R-3.5}\rm
Assume that $f\in\cF_\oc^+(0,\infty)\setminus\bigcup_{\alpha\in[0,\infty]}F_\alpha$. Then
in the proof of Theorem \ref{T-3.2}, we gave an extremal decomposition \eqref{F-3.11} for
$f$. We here remark that $f$ admits continuously many different extremal decompositions.
In fact, there is an $\alpha_1\in(0,\infty]$ such that the tangential line of $f$ at
$x=\alpha_1$ passes the origin (take $\alpha_1=\infty$ when a straight line from the
origin is tangential to $f$ at infinity). It is clear that $\alpha_0<\alpha_1$ where
$\alpha_0$ is as given in the proof of Theorem \ref{T-3.2}. Then for every
$\alpha\in[\alpha_0,\alpha_1]$, let $a_\alpha+c_\alpha x$ be the tangential line of $f$
at $x=\alpha$. Then $a_\alpha\ge0$, $c_\alpha\ge0$, and $f-(a_\alpha+c_\alpha x)$ belongs
to $F_\alpha$. Hence $f$ admits an extremal decomposition
\begin{equation}\label{F-3.13}
f(x)=a_\alpha+c_\alpha x+b_\alpha(x-\alpha)^2
+\int_{[0,\infty)}{(x-\alpha)^2(1+\alpha+\lambda)\over x+\lambda}\,d\nu_\alpha(\lambda),
\end{equation}
where $b_\alpha\ge0$ and $\nu_\alpha$ is a positive finite measure on $[0,\infty)$.
For the function \eqref{F-3.12}, $1+x^2/(x+1)$ is the decomposition corresponding to
$\alpha=\alpha_0=0$ and $x+1/(x+1)$ is the decomposition corresponding to
$\alpha=\alpha_1=\infty$ ($x$ is tangential to $x+1/(x+1)$ at infinity). Furthermore,
for each $\alpha\in[0,\infty]$ we have an extremal decomposition
$$
x+{1\over x+1}={2\alpha+1\over(\alpha+1)^2}+{\alpha^2+2\alpha\over(\alpha+1)^2}\,x
+{(x-\alpha)^2\over(\alpha+1)^2(x+1)}
$$
of the form \eqref{F-3.13}. The function
$$
2x^2-2x+1=2(x-\alpha)^2+2(2\alpha-1)x+1-2\alpha^2
$$
gives another example. The above right-hand side expression gives an extremal decomposition
of the form \eqref{F-3.13} for each $\alpha\in[1/2,1/\sqrt2]$. But the above function has
even more different extremal decompositions such as $x^2+(x-1)^2$. In this way, there are
many extremal decompositions for an element of
$\cF_\oc^+(0,\infty)\setminus\bigcup_{\alpha\in[0,\infty]}F_\alpha$.
\end{remark}

\begin{remark}\label{r-3.7}\rm
The transformation $\tau:f\mapsto xf(x^{-1})$ is a linear homeomorphism of $\cF(0,\infty)$
onto itself, which maps $\cF^+(0,\infty)$ onto itself. Since
$\tau(\ffi_\lambda)=\psi_{\lambda^{-1}}$ for every $\lambda\in[0,\infty]$ (where
$\ffi_\alpha,\psi_\alpha$ were defined in the proof of Theorem \ref{T-3.3}), we see that
$\tau$ maps $F_\infty$ onto $F_0$ (and vice versa). This and \eqref{F-3.6} show that the
structure of extreme elements as well as the facial cone structure is completely
equivalent among $F_\infty$, $F_0$ and
$\{f\in\cF^+(0,\infty):\mbox{$f$ is operator monotone}\}$ via the simple isomorphisms
$f\leftrightarrow xf(x^{-1})\leftrightarrow f(x^{-1})$. Incidentally, we see that the
following conditions are equivalent for $f\in\cF^+(0,\infty)$:
\begin{itemize}
\item[(a)] $f$ is operator monotone decreasing (or $f(x^{-1})$ is operator monotone);
\item[(b)] $f$ is operator log-convex in the sense of \cite{AH};
\item[(c)] $f$ is operator convex and non-increasing in the numerical sense, i.e.,
$f\in F_\infty$;
\item[(d)] $xf(x^{-1})$ is operator convex and $\lim_{x\to\infty}f(x)/x=0$, i.e,
$xf(x^{-1})$ is in $F_0$.
\end{itemize}
(indeed, (a) $\Leftrightarrow$ (b) $\Leftrightarrow$ (c) was shown in
\cite[Theorems 2.1 and 3.1]{AH} and (a) $\Leftrightarrow$ (d) is immediately seen from
\cite[Theorem 2.4]{HP}.) It is also obvious that $F_\alpha$ is transformed onto $F_\beta$
by a simple isomorphism $f\mapsto f((\alpha/\beta)x)$ of $\cF(0,\infty)$ for any
$\alpha,\beta\in(0,\infty)$. However, it does not seem that we have such a simple
isomorphism between $F_0$ and $F_\alpha$, $0<\alpha<\infty$.
\end{remark}

\section{Facial subcones}

In the previous section we proved that the extreme rays of $\cF_\oc^+(0,\infty)$ are given
by the elements $g_{\alpha,\lambda}$ ($\alpha,\lambda\in[0,\infty]$) defined as
\begin{equation}\label{F-4.1}
g_{\alpha,\lambda}(x):=\begin{cases}    
{(x-\alpha)^2\over x+\lambda} & \text{($0\le\alpha<\infty$, $0\le\lambda<\infty$)}, \\
(x-\alpha)^2 & \text{($0\le\alpha<\infty$, $\lambda=\infty$)}, \\                                                                                                          
1\over x+\lambda & \text{($\alpha=\infty$, $0\le\lambda<\infty$)}, \\
\,1 & \text{($\alpha=\infty$, $\lambda=\infty$)}.
\end{cases}
\end{equation}
Furthermore, $\cF_\oc^+(0,\infty)$ is the closed convex cone generated by the above extreme
elements. For each $\alpha\in[0,\infty]$ and each nonempty closed subset $\Lambda$ of
$[0,\infty]$, we denote by $F_{\alpha,\Lambda}$ the closed convex cone generated by
$\{g_{\alpha,\lambda}:\lambda\in\Lambda\}$. Moreover, for each closed subset $\Lambda$
(possibly $\emptyset$) of $[0,\infty]$, we denote by $E_\Lambda$ the closed convex cone
generated by $\{g_{\alpha,\lambda}:\alpha\in[0,\infty],\,\lambda\in\Lambda\}$ and
$\{g_{\infty,\infty},g_{0,0}\}=\{1,x\}$. In particular,
\begin{equation}\label{F-4.2}
F_{\infty,\{\infty\}}=\{\lambda 1:\lambda\ge0\},\qquad
F_{0,\{0\}}=\{\lambda x:\lambda\ge0\}.
\end{equation}
The closed faces $F_\alpha$ ($0\le\alpha\le\infty$) introduced in Theorem \ref{T-3.3}
coincide with $F_{\alpha,[0,\infty]}$ so that
\begin{align*}
F_{\alpha,[0,\infty]}
&=F_\alpha=\{f\in\cF_\oc^+(0,\infty):f(\alpha)=0\}\quad(0\le\alpha<\infty), \\
F_{\infty,[0,\infty]}
&=F_\infty=\{f\in\cF_\oc^+(0,\infty):\mbox{$f$ is non-increasing}\}.
\end{align*}

Note that the extremal integral decomposition of an element of $\cF_\oc^+(0,\infty)$ is not
unique since $\cF_\oc^+(0,\infty)$ is not a simplicial convex cone (Theorem \ref{T-3.2}).
But Theorem \ref{T-2.2} gives a standard integral representation for a general operator
convex function on $(0,\infty)$ (hence for $f\in\cF_\oc^+(0,\infty)$). By taking $\alpha=1$
we rewrite the representation as follows:
$$
f(x)=f(1)+f'(1)(x-1)+\gamma(x-1)^2
+\int_{[0,\infty)}{(x-1)^2(2+\lambda)\over x+\lambda}\,d\nu(\lambda),
$$
where $\gamma\ge0$ and $\nu$ is a finite positive measure on $[0,\infty)$ transformed from
the representing measure $\mu$ of $f$ by
$$
d\nu(\lambda):={1\over(1+\lambda)^2(2+\lambda)}\,d\mu(\lambda).
$$
Furthermore, by extending $\nu$ to a measure on $[0,\infty]$ with $\nu(\{\infty\})=\gamma$,
one can write the above representation as
\begin{equation}\label{F-4.3}
f(x)=f(1)+f'(1)(x-1)+\int_{[0,\infty]}{(x-1)^2(2+\lambda)\over x+\lambda}\,d\nu(\lambda),
\end{equation}
where $(x-1)^2(2+\lambda)/(x+\lambda)$ means $(x-1)^2$ when $\lambda=\infty$. Now, for
every $f\in\cF_\oc^+(0,\infty)$ let $\supp\nu$ be the support of $\nu$, and we denote it
by $\Sigma_f$. From the above construction, $\Sigma_f$ is given in terms of the canonical
data $\mu$ and $\gamma$ of $f$ as follows:
\begin{equation}\label{F-4.4}
\Sigma_f=\begin{cases}
\supp\mu & \text {if $\gamma=0$}, \\
\supp\mu\cup\{\infty\} & \text{if $\gamma>0$}.
\end{cases}
\end{equation}
Moreover, for every closed subset $\Lambda$ of $[0,\infty]$ we define
$$
\cF_\Lambda:=\{f\in\cF_\oc^+(0,\infty):\Sigma_f\subset\Lambda\}.
$$
In particular, since $\Sigma_f=\emptyset$ if and only if $f(x)=p+qx$ with $p,q\ge0$, we
have
\begin{equation}\label{F-4.5}
\cF_\emptyset=\{p+qx:p,q\ge0\}.
\end{equation}

The aim of this section is to prove the next theorem determining all closed faces of
$\cF_\oc^+(0,\infty)$.

\begin{theorem}\label{T-4.1}
The family of all closed faces ($\ne\{0\}$) of $\cF_\oc^+(0,\infty)$ is given by
\begin{itemize}
\item[(i)] $F_{\alpha,\Lambda}$ for $\alpha\in[0,\infty]$ and for nonempty closed subsets
$\Lambda$ of $[0,\infty]$,
\item[(ii)] $E_\Lambda$ for closed subsets $\Lambda$ of $[0,\infty]$.
\end{itemize}

Furthermore, the family of maximal closed faces of $\cF_\oc^+(0,\infty)$ is given by
$F_\alpha$ for $\alpha\in[0,\infty]$, where a closed face $F$ of $\cF_\oc^+(0,\infty)$ is
said to be maximal if $F\ne\cF_\oc^+(0,\infty)$ and there is no closed face $\widetilde F$
of $\cF_\oc^+(0,\infty)$ such that
$F\subsetneqq\widetilde F\subsetneqq\cF_\oc^+(0,\infty)$.
\end{theorem}

\begin{remark}\label{R-4.2}\rm
Under the transformation $\tau:f\mapsto xf(x^{-1})$ on $\cF^+(0,\infty)$ mentioned in
Remark \ref{r-3.7}, the closed faces in Theorem \ref{T-4.1} are related as follows:
$$
\tau(F_{\alpha,\Lambda})=F_{\alpha^{-1},\Lambda^{-1}},\qquad
\tau(E_\Lambda)=E_{\Lambda^{-1}}.
$$
In particular, $\tau(F_{0,\{0\}})=F_{\infty,\{\infty\}}$ for those in \eqref{F-4.2},
and $\tau(E_\emptyset)=E_\emptyset$ for $E_\emptyset=\cF_\emptyset$ in \eqref{F-4.5}.
\end{remark}

\begin{lemma}\label{L-4.3}
For any closed subset $\Lambda$ (possibly $\emptyset$) of $[0,\infty]$, $\cF_\Lambda$ is
a closed face of $\cF_\oc^+(0,\infty)$.
\end{lemma}

\begin{proof}
Let $f_1,f_2$ be in $\cF_\oc^+(0,\infty)$, having the representations as in \eqref{F-4.3}
with measures $\nu_1,\nu_2$, respectively. Then $f_1+f_2$ has the representation with
measure $\nu_1+\nu_2$. If $f_1,f_2\in\cF_\Lambda$, then
$\supp\nu_i\subset\Lambda$
for $i=1,2$ so that $\supp(\nu_1+\nu_2)\subset\Lambda$ and hence $f_1+f_2\in\cF_\Lambda$. 
So $\cF_\Lambda$ is a convex cone. If $f_1+f_2\in\cF_\Lambda$, then
$\supp(\nu_1+\nu_2)\subset\Lambda$ so that $\supp\nu_i\subset\Lambda$, $i=1,2$, and hence
$f_1,f_2\in\cF_\Lambda$. So $\cF_\Lambda$ is a face of $\cF_\oc^+(0,\infty)$. It remains
to prove that $\cF_\Lambda$ is closed (in the pointwise topology on $\cF_\oc^+(0,\infty)$).
Recall that $\cF_\oc^+(0,\infty)$ is metrizable, and let $\{f_n\}$ be a sequence in
$\cF_\Lambda$ such that $f_n\to f\in\cF_\oc^+(0,\infty)$ in the pointwise topology. So,
each $f_n$ has the representation \eqref{F-4.3} with measure $\nu_n$, and $f$ does so with
$\nu$. We have $\supp\nu_n\subset\Lambda$ for all $n$. For any $\eps>0$ one can choose a
$\delta\in(0,1)$ such that
$$
f'(1)-\eps<{f(1)-f(1-\delta)\over\delta}
\le{f(1+\delta)-f(1)\over\delta}<f'(1)+\eps.
$$
Hence, if $n$ is sufficiently large, then
$$
f'(1)-\eps<{f_n(1)-f_n(1-\delta)\over\delta}
\le{f_n(1+\delta)-f_n(1)\over\delta}<f'(1)+\eps
$$
so that $f'(1)-\eps<f_n'(1)<f'(1)+\eps$. This shows that $f_n'(1)\to f'(1)$ as
$n\to\infty$. Hence
\begin{equation}\label{F-4.6}
\int_{[0,\infty]}{2+\lambda\over x+\lambda}\,d\nu(\lambda)
=\lim_{n\to\infty}\int_{[0,\infty]}{2+\lambda\over x+\lambda}\,d\nu_n(\lambda),
\qquad x\in(0,\infty)\setminus\{1\},
\end{equation}
where $(2+\lambda)/(x+\lambda)=1$ for $\lambda=\infty$. Moreover, since
$$
f_n(2)=f_n(1)+f_n'(1)+\nu_n([0,\infty])\longrightarrow
f(2)=f(1)+f'(1)+\nu([0,\infty]),
$$
we have $\nu([0,\infty])=\lim_{n\to\infty}\nu_n([0,\infty])$. We may consider $\nu_n$ and
$\nu$ as elements of the dual Banach space $C([0,\infty])^*$ of the Banach space
$C([0,\infty])$ with sup-norm. Since the weak* topology on a norm-bounded subset of
$C([0,\infty])^*$ is metrizable, one can choose, by Alaoglu's theorem, a subsequence
$\{\nu_{n(k)}\}$ converging to a $\nu_0\in C([0,\infty])^*$ (identified with a finite
positive Borel measure on $[0,\infty]$) as $k\to\infty$ in the weak* topology. For each
$x\in(0,\infty)$, since $\lambda\mapsto(2+\lambda)/(x+\lambda)$ is in $C([0,\infty])$,
this weak* convergence and \eqref{F-4.6} imply that
$$
\int_{[0,\infty]}{2+\lambda\over x+\lambda}\,d\nu(\lambda)
=\int_{[0,\infty]}{2+\lambda\over x+\lambda}\,d\nu_0(\lambda),
\qquad x\in(0,\infty)\setminus\{1\},
$$
and in particular when $x=2$, we have $\nu([0,\infty])=\nu_0([0,\infty])$. Moreover, since
$(2+\lambda)/(x+\lambda)=1+(2-x)/(x+\lambda)$, we have
$$
\int_{[0,\infty]}{1\over x+\lambda}\,d\nu(\lambda)
=\int_{[0,\infty]}{1\over x+\lambda}\,d\nu_0(\lambda),
\qquad x\in(0,\infty)\setminus\{1,2\}.
$$
In fact, this holds for $x=1,2$ as well by taking limits. Noting that
$$
{1\over x_1+\lambda}\cdot{1\over x_2+\lambda}={1\over x_2-x_1}
\biggl({1\over x_1+\lambda}-{1\over x_2+\lambda}\biggr),
\qquad x_1,x_2\in(0,\infty),\ x_1\ne x_2,
$$
we see by the Stone-Weierstrass theorem that the linear span of $1$ and
$\lambda\mapsto1/(x+\lambda)$ for $x\in(0,\infty)$ is a norm-dense subalgebra of
$C([0,\infty])$. Therefore, $\nu=\nu_0$. Here, $\nu_n([0,\infty]\setminus\Lambda)=0$ for
all $n$. Since $\mu\mapsto\mu([0,\infty]\setminus\Lambda)$ is lower semicontinuous in the
weak* topology on the set of finite positive Borel measures on $[0,\infty]$, it follows
that $\nu([0,\infty]\setminus\Lambda)=0$ and so $f\in\cF_\Lambda$. Hence $\cF_\Lambda$ is
closed.
\end{proof}

In the following lemmas we clarify the structure of $F_{\alpha,\Lambda}$ and $E_\Lambda$
in Theorem \ref{T-4.1} in terms of $\cF_\Lambda$, which will be also useful to prove the
theorem.

\begin{lemma}\label{L-4.4}
Let $\alpha\in(0,\infty)$ and $\Lambda$ be a nonempty closed subset of $[0,\infty]$.
\begin{itemize}
\item[(1)] For every $\lambda\in[0,\infty]$, $g_{\alpha,\lambda}\in\cF_\Lambda$ if and
only if $\lambda\in\Lambda$.
\item[(2)] $F_{\alpha,\Lambda}=F_\alpha\cap\cF_\Lambda$ and hence $F_{\alpha,\Lambda}$ is
a closed face of $\cF_\oc^+(0,\infty)$.
\item[(3)] The correspondence $\Lambda\mapsto F_{\alpha,\Lambda}$ is one-to-one for
nonempty closed subsets $\Lambda$ of $[0,\infty]$.
\end{itemize}
\end{lemma}

\begin{proof}
(1)\enspace
Note that
\begin{align*}
{(x-\alpha)^2\over x+\lambda}&={(1-\alpha)^2\over1+\lambda}
+{(1-\alpha)(1+\alpha+2\lambda)\over(1+\lambda)^2}\,(x-1)
+\biggl({\alpha+\lambda\over1+\lambda}\biggr)^2{(x-1)^2\over x+\lambda}, \\
(x-\alpha)^2&=(\alpha-1)^2-2(\alpha-1)(x-1)+(x-1)^2.
\end{align*}
From these expressions and the uniqueness of the representation \eqref{F-4.3} we have
$g_{\alpha,\lambda}\in\cF_{\{\lambda\}}$ for every $\lambda\in[0,\infty]$. (This is also
immediate from Theorem \ref{T-2.2} and \eqref{F-4.4}.) The assertion is then obvious.

(2)\enspace
Since both $F_\alpha$ and $\cF_\Lambda$ are closed faces of  $\cF_\oc^+(0,\infty)$ due to
Theorem \ref{T-3.3} and Lemma \ref{L-4.3}, $F_\alpha\cap\cF_\Lambda$ is also a closed face
of $\cF_\oc^+(0,\infty)$. By definition it is obvious that
$F_{\alpha,\Lambda}\subset F_\alpha\cap\cF_\Lambda$. To prove the converse, define
$$
K:=\{f\in F_\alpha\cap\cF_\Lambda:f(\alpha+1)=1\},
$$
which is a compact convex subset of $\cF(0,\infty)$ (see Theorem \ref{T-3.3} and its
proof). The Krein-Milman theorem says that $K=\overline{\conv}(\ex K)$, the closed convex
hull of the set $\ex K$ of all extreme points of $K$. Since $K$ becomes a face of
$\{f\in\cF_\oc^+(0,\infty):f(\alpha+1)=1\}$, every extreme point of $K$ belongs to an
extreme ray of $\cF_\oc^+(0,\infty)$. Hence by Theorem \ref{T-3.2} and (1) above we have
$$
\ex K=\biggl\{{1\over g_{\alpha,\lambda}(\alpha+1)}\,g_{\alpha,\lambda}:
\lambda\in\Lambda\biggr\}.
$$
This shows that $F_\alpha\cap\cF_\Lambda$ is generated by
$\{g_{\alpha,\lambda}:\lambda\in\Lambda\}$, so
$F_\alpha\cap\cF_\Lambda\subset F_{\alpha,\Lambda}$.

(3) immediately follows from (1) and (2).
\end{proof}

\begin{lemma}\label{L-4.5}
Let $\alpha=0$ or $\infty$. For any nonempty closed subset $\Lambda$ of $[0,\infty]$,
$F_{\alpha,\Lambda}$ is a closed face of $\cF_\oc^+(0,\infty)$, and furthermore
$\Lambda\mapsto F_{\alpha,\Lambda}$ is one-to-one.
\end{lemma}

\begin{proof}
To prove the lemma, it is convenient to use the integral representations \eqref{F-3.7} for
$\alpha=\infty$ and \eqref{F-3.9} for $\alpha=0$ instead of \eqref{F-4.3}.

First, assume that $\alpha=0$. For $f\in F_0$ we rewrite \eqref{F-3.9} as
\begin{equation}\label{F-4.7}
f(x)=\int_{[0,\infty]}{x^2(1+\lambda)\over x+\lambda}\,d\mu(\lambda),
\end{equation}
where $\mu$ is a unique finite positive measure on $[0,\infty]$ and
$x^2(1+\lambda)/(x+\lambda)$ means $x^2$ for $\lambda=\infty$. For every nonempty closed
$\Lambda\subset[0,\infty]$ we define $\cF_\Lambda^{(0)}$ to be the set of
$f\in\cF_\oc^+(0,\infty)$ whose representing measure $\mu$ in \eqref{F-4.7} is supported
in $\Lambda$. Then, similarly to Lemma \ref{L-4.3} it is seen that $\cF_\Lambda^{(0)}$ is
a closed face of $\cF_\oc^+(0,\infty)$. Moreover, the statements of Lemma \ref{L-4.4}
can be verified in this case as well when $\cF_\Lambda$ is replaced with
$\cF_\Lambda^{(0)}$. The injectivity of $\Lambda\mapsto F_{0,\Lambda}$ is also seen since
$g_{0,\lambda}\in\cF_\Lambda^{(0)}$ if and only if $\lambda\in\Lambda$.

Secondly, assume that $\alpha=\infty$. For $f\in F_\infty$ we rewrite \eqref{F-3.7} as
\begin{equation}\label{F-4.8}
f(x)=\int_{[0,\infty]}{1+\lambda\over x+\lambda}\,d\nu(\lambda),
\end{equation}
where $\nu$ is a unique finite positive measure on $[0,\infty]$ and
$(1+\lambda)/(x+\lambda)$ means $1$ for $\lambda=\infty$. Then the assertion can be shown
similarly to the case $\alpha=0$ above by defining $\cF_\Lambda^{(\infty)}$ to be the set
of $f\in\cF_\oc^+(0,\infty)$ whose representing measure $\nu$ in \eqref{F-4.8} is supported
in $\Lambda$.
\end{proof}

\begin{lemma}\label{L-4.6}
For every closed subset $\Lambda$ of $[0,\infty]$, $E_\Lambda=\cF_\Lambda$ and hence
$E_\Lambda$ is a closed face of $\cF_\oc^+(0,\infty)$. Furthermore,
$\Lambda\mapsto E_\Lambda$ is one-to-one.
\end{lemma}

\begin{proof}
Lemma \ref{L-4.4}\,(1) says that $g_{\alpha,\lambda}\in\cF_\Lambda$ if
$\alpha\in(0,\infty)$ and $\lambda\in\Lambda$. If $\lambda\in\Lambda\setminus\{0\}$, then
\begin{equation}\label{F-4.9}
g_{0,\lambda}(x)={1\over1+\lambda}+{1+2\lambda\over(1+\lambda)^2}\,(x-1)
+\biggl({\lambda\over1+\lambda}\biggr)^2{(x-1)^2\over x+\lambda}
\end{equation}
so that $g_{0,\lambda}\in\cF_\Lambda$. Also, if $\lambda\in\Lambda\setminus\{\infty\}$,
then
\begin{equation}\label{F-4.10}
g_{\infty,\lambda}(x)={1\over1+\lambda}-{1\over(1+\lambda)^2}\,(x-1)
+{1\over(1+\lambda)^2}\cdot{(x-1)^2\over x+\lambda}
\end{equation}
so that $g_{\infty,\lambda}\in\cF_\Lambda$. Furthermore,
$g_{0,0},g_{\infty,\infty}\in\cF_\Lambda$ (see \eqref{F-4.5}), and hence we have
$E_\Lambda\subset\cF_\Lambda$. To prove the converse, define
$$
K:=\{f\in\cF_\Lambda:f(1/2)+f(2)=1\}.
$$
Let $f\in K$ be expressed as in \eqref{F-4.3}. Since $f(1/2)\le1$ and $f(2)\le1$, we have
$$
f(1)\in[0,1],\quad f'(1)\in[-2,1],\quad
f(1)+f'(1)+\nu([0,\infty])=f(2)\in[0,1],
$$
which yield that
$$
f(x)\le3+x+3(x-1)^2\max\{1,2/x\},\qquad x\in(0,\infty).
$$
Hence we see that $K$ is a compact convex subset of $\cF(0,\infty)$. By the Krein-Milman
theorem we have $K=\overline{\conv}(\ex K)$. Since $K$ becomes a face of
$\{f\in\cF_\oc^+(0,\infty):f(1/2)+f(2)=1\}$ due to Lemma \ref{L-4.3}, every extreme point
of $K$ belongs to an extreme ray of $\cF_\oc^+(0,\infty)$. Thus, by Theorem \ref{T-3.2}
and Lemma \ref{L-4.4}\,(1) together with \eqref{F-4.9} and \eqref{F-4.10} we see that
\begin{align*}
\ex K&\subset
\biggl\{{1\over g_{\alpha,\lambda}(1/2)+g_{\alpha,\lambda}(2)}\,g_{\alpha,\lambda}:
\alpha\in[0,\infty],\ \lambda\in\Lambda\biggr\} \\
&\qquad\cup\biggl\{{1\over g_{0,0}(1/2)+g_{0,0}(2)}\,g_{0,0},
{1\over g_{\infty,\infty}(1/2)+g_{\infty,\infty}(2)}\,g_{\infty,\infty}\biggr\}.
\end{align*}
This shows that $\cF_\Lambda$ is generated by $\{g_{\alpha,\lambda}:\alpha\in[0,\infty],\,
\lambda\in\Lambda\}\cup\{g_{0,0},g_{\infty,\infty}\}$, so $\cF_\Lambda\subset E_\Lambda$
by definition of $E_\Lambda$. Therefore, $E_\Lambda=\cF_\Lambda$ and also the injectivity
of $\Lambda\mapsto E_\Lambda$ follows from Lemma \ref{L-4.4}\,(1).
\end{proof}

\noindent
{\it Proof of Theorem \ref{T-4.1}.}\enspace
By Lemmas \ref{L-4.4}--\ref{L-4.6} all of $F_{\alpha,\Lambda}$ and $E_\Lambda$ given in
(i) and (ii) of the theorem are closed faces of $\cF_\oc^+(0,\infty)$. So it remains to
prove that if $F$ ($\ne\{0\}$) is a closed face of $\cF_\oc^+(0,\infty)$, then either
$F=F_{\alpha,\Lambda}$ for some $\alpha\in[0,\infty]$ and some nonempty closed
$\Lambda\subset[0,\infty]$ or $F=E_\Lambda$ for some closed $\Lambda\subset[0,\infty]$.

(1)\enspace
First, assume that $F\subset F_\alpha$ for some $\alpha\in[0,\infty]$, and prove that
$F=F_{\alpha,\Lambda}$ for some nonempty closed $\Lambda\subset[0,\infty]$. Define
$$
K:=\begin{cases}
\{f\in F:f(\alpha+1)=1\} & \text{if $\alpha\in[0,\infty)$}, \\
\{f\in F:f(1)=1\} & \text{if $\alpha=\infty$}.
\end{cases}
$$
Then $K$ is a compact convex subset of $F_\alpha$ so that $K=\overline{\conv}(\ex K)$.
Define $\Lambda:=\{\lambda\in[0,\infty]:g_{\alpha,\lambda}\in F\}$. Since
$$
\psi_{\alpha,\lambda}:=\begin{cases}
(1+\lambda)g_{\alpha,\lambda} & \text{if $\lambda\in[0,\infty)$}, \\
g_{\alpha,\infty} & \text{if $\lambda=\infty$}
\end{cases}
$$
is continuous in $\lambda\in[0,\infty]$ in the pointwise topology (see \eqref{F-4.1}),
it follows that $\Lambda$ is a closed subset of $[0,\infty]$. Since $K$ is a face of
$F_\alpha^0$ where $F_\alpha^0$ is given in \eqref{F-3.4} and \eqref{F-3.5}, it is seen
as in the proof of Lemma \ref{L-4.4} that
$$
\ex K=\begin{cases}
\Bigl\{{1\over g_{\alpha,\lambda}(\alpha+1)}\,g_{\alpha,\lambda}:
\lambda\in\Lambda\Bigr\} & \text{if $\alpha\in[0,\infty)$}, \\
\Bigl\{{1\over g_{\alpha,\lambda}(1)}\,g_{\alpha,\lambda}:
\lambda\in\Lambda\Bigr\} & \text{if $\alpha=\infty$}.
\end{cases}
$$
This shows that $F=F_{\alpha,\Lambda}$ with $\Lambda\ne\emptyset$.

(2)\enspace
Next, assume that $F\not\subset F_\alpha$ for every $\alpha\in[0,\infty]$, and prove that
$F=E_\Lambda$ for some closed $\Lambda\subset[0,\infty]$. Note that at least
one of the following two cases holds:
\begin{itemize}
\item[(a)] there is an $f_0\in F$ such that
$f_0\notin\bigcup_{\alpha\in[0,\infty]}F_\alpha$,
\item[(b)] there are $f_1,f_2\in F$ such that $f_1\in F_{\alpha_1}$ and
$f_2\in F_{\alpha_2}$ with $\alpha_1,\alpha_2\in[0,\infty]$, $\alpha_1\ne\alpha_2$.
\end{itemize}
In the case (a), we have $f_0(\alpha_0)=\min_{x\in[0,\infty]}f_0(x)>0$ for some
$\alpha_0\in[0,\infty]$, where $f(\alpha_0)$ means $f(+0)$ or $\lim_{x\to\infty}f(x)$
for $\alpha_0=0$ or $\infty$, respectively. In the case (b), let $f_0:=f_1+f_2\in F$;
then the situation is same as in the case (a). Hence, in each case, there are $p,q>0$
such that $f_0(x)\ge p+qx$ for all $x\in(0,\infty)$. Since $F$ is a face of
$\cF_\oc^+(0,\infty)$, writing
$$
f_0(x)=p+qx+(f_0(x)-p-qx),
$$
we see that $1,x\in F$. Now assume that $F\ne E_\emptyset$ ($=\cF_\emptyset$ in
\eqref{F-4.5}). We prove that if $(\alpha,\lambda)$ satisfies one of
\begin{itemize}
\item[(i)] $\alpha\in(0,\infty)$ and $\lambda\in[0,\infty]$,
\item[(ii)] $\alpha=0$ and $\lambda\in(0,\infty]$,
\item[(iii)] $\alpha=\infty$ and $\lambda\in[0,\infty)$,
\end{itemize}
and if $g_{\alpha,\lambda}\in F$, then $g_{\beta,\lambda}\in F$ for all
$\beta\in[0,\infty]\setminus\{\alpha\}$. Define
\begin{equation}\label{F-4.11}
f_\beta(x):=g_{\alpha,\lambda}(x)
-\{g_{\alpha,\lambda}'(\beta)(x-\beta)+g_{\alpha,\lambda}(\beta)\}.
\end{equation}

First, assume (i). If $\beta>\alpha$, then $g_{\alpha,\lambda}'(\beta)>0$ and
$$
g_{\alpha,\lambda}(x)+(g_{\alpha,\lambda}'(\beta)\beta)1
=f_\beta(x)+g_{\alpha,\lambda}'(\beta)x+g_{\alpha,\lambda}(\beta)1.
$$
Since the above left-hand side belongs to $F$ and $f_\beta\in\cF_\oc^+(0,\infty)$, we see
that $f_\beta\in F$. If $0<\beta<\alpha$, then $g_{\alpha,\lambda}'(\beta)<0$ and
$$
g_{\alpha,\lambda}(x)+(-g_{\alpha,\lambda}'(\beta))x
=f_\beta(x)+(-g_{\alpha,\lambda}'(\beta)\beta)1+g_{\alpha,\lambda}(\beta)1.
$$
Since the above left-hand side belongs to $F$, we have $f_\beta\in F$ again. A direct
computation gives
\begin{equation}\label{F-4.12}
f_\beta=\biggl({\alpha+\lambda\over\beta+\lambda}\biggr)^2g_{\beta,\lambda}.
\end{equation}
Hence it follows that $g_{\beta,\lambda}\in F$ for all $\beta\in(0,\infty)$. We have
$g_{0,\lambda},g_{\infty,\lambda}\in F$ as well by taking limits
$g_{0,\lambda}=\lim_{\beta\searrow0}g_{\beta,\lambda}$ and
$g_{\infty,\lambda}=\lim_{\beta\to\infty}\beta^{-2}g_{\beta,\lambda}$. Next, assume (ii).
For any $\beta\in(0,\infty)$, $f_\beta$ given in \eqref{F-4.11} with $\alpha=0$ is in $F$
as above and \eqref{F-4.12} becomes $f_\beta=(\lambda/(\beta+\lambda))^2g_{\beta,\lambda}$
so that $g_{\beta,\lambda}\in F$, and also $g_{\infty,\lambda}\in F$ by taking limit.
Finally, assume (iii). For any $\beta\in(0,\infty)$, $f_\beta$ in \eqref{F-4.11} with
$\alpha=\infty$ is in $F$ as above and \eqref{F-4.12} becomes
$f_\beta=(1/(\beta+\lambda))^2g_{\beta,\lambda}$ so that $g_{\beta,\lambda}\in F$, and
$g_{0,\lambda}\in F$ by taking limit.

Now, we define $\Lambda:=\{\lambda\in[0,\infty]:g_{1,\lambda}\in F\}$, which is a closed
subset of $[0,\infty]$ as in the above proof of (1). From the facts proved above for the
cases (i)--(iii), we see that
\begin{align}
\{\lambda\in[0,\infty]:g_{\alpha,\lambda}\in F\}&=\Lambda\quad
\mbox{for all $\alpha\in(0,\infty)$}, \label{F-4.13} \\
\{\lambda\in[0,\infty]:g_{0,\lambda}\in F\}&=\Lambda\cup\{0\}, \label{F-4.14}\\
\{\lambda\in[0,\infty]:g_{\infty,\lambda}\in F\}&=\Lambda\cup\{\infty\}. \label{F-4.15}
\end{align}
Then it is obvious that $E_\Lambda\subset F$. Since $F$ is a closed face of
$\cF_\oc^+(0,\infty)$, it is seen as in the proof of Lemma \ref{L-4.6} that $F$ is
generated by
$\{g_{\alpha,\lambda}:\alpha,\lambda\in[0,\infty],\,g_{\alpha,\lambda}\in F\}$. Hence,
to prove that $F\subset E_\Lambda$, it suffices to show that if $g_{\alpha,\lambda}\in F$
then $g_{\alpha,\lambda}\in E_\Lambda$. Since $g_{0,0},g_{\infty,\infty}\in E_\Lambda$,
we may show this if $(\alpha,\lambda)$ satisfies one of (i)--(iii) above. But in these
cases, the assertion is obvious from \eqref{F-4.13}--\eqref{F-4.15}. Therefore,
$F=E_\Lambda$ and $\Lambda\ne\emptyset$ (thanks to the assumption $F\ne\cF_\emptyset$).

Finally, we prove the latter statement of Theorem \ref{T-4.1}. If $\Lambda\ne[0,\infty]$,
then it follows from the injectivity assertions in Lemmas \ref{L-4.4}--\ref{L-4.6} that
$F_{\alpha,\Lambda}$ and $E_\Lambda$ are not maximal closed faces of $\cF_\oc^+(0,\infty)$.
When $\Lambda=[0,\infty]$, $E_\Lambda=\cF_\oc^+(0,\infty)$. On the other hand, one can
easily see from the above proof of (2) that each $F_\alpha$ is a maximal closed face of
$\cF_\oc^+(0,\infty)$. Hence the maximal closed faces of $\cF_\oc^+(0,\infty)$ are only
$F_\alpha$'s.\qed

\medskip
We finally determine all simplicial closed faces of $\cF_\oc^+(0,\infty)$.

\begin{corollary}\label{C-4.7}
The family of simplicial closed faces of $\cF_\oc^+(0,\infty)$ are $F_{\alpha,\Lambda}$
for $\alpha\in[0,\infty]$ and for nonempty closed subsets $\Lambda$ of $[0,\infty]$.
\end{corollary}

\begin{proof}
By Theorem \ref{T-3.3}, $F_\alpha$'s are simplicial closed convex cones of
$\cF_\oc^+(0,\infty)$. Since $F_{\alpha,\Lambda}$ is a closed face of $F_\alpha$ by
Theorem \ref{T-4.1}, $F_{\alpha,\Lambda}$ is a simplicial closed face of
$\cF_\oc^+(0,\infty)$. Next, to show that $E_\Lambda$'s are not simplicial, consider the
identity
$$
(x-1)^2+(x-2)^2=2\biggl(x-{3\over2}\biggr)^2+{1\over2},
\ \ \mbox{i.e.,}
\ \ g_{1,\infty}+g_{2,\infty}=2g_{3/2,\infty}+{1\over2}g_{\infty,\infty}.
$$
When $\infty\in\Lambda$, the above shows that $E_\Lambda$ is not simplicial. When
$\Lambda$ contains a  $\lambda\in[0,\infty)$, then the same holds since
$$
{(x-1)^2\over x+\lambda}+{(x-2)^2\over x+\lambda}
=2\,{\bigl(x-{3\over2}\bigr)^2\over x+\lambda}+{1\over2}\cdot{1\over x+\lambda},
\ \ \mbox{i.e.,}
\ \ g_{1,\lambda}+g_{2,\lambda}=2g_{3/2,\lambda}+{1\over2}g_{\infty,\lambda}.
$$
\end{proof}

\section{Operator convex functions on $(-1,1)$}

In this section we consider the non-negative operator convex functions on a finite open
interval. Operator convex functions on any finite open interval $(a,b)$ are transformed
into those on $(-1,1)$ via the affine function $x\in(-1,1)\mapsto((b-a)/2)x+(b+a)/2$.
Therefore, we may assume without loss of generality that the interval is $(-1,1)$. We
write $\cF(-1,1)$ for the locally convex topological space of all real functions on
$(-1,1)$ with the pointwise convergence topology. Let $\cF_\oc^+(-1,1)$ be the closed
convex cone in $\cF(-1,1)$ consisting of all operator convex functions $f\ge0$ on $(-1,1)$,
and $\cF_\oc^{++}(-1,1)$ be the convex cone of all operator convex functions $f>0$ on
$(-1,1)$.

For each $\alpha\in[-1,1]$ define
\begin{equation}\label{F-5.1}
\widetilde F_\alpha:=\{f\in\cF_\oc^+(-1,1):f(\alpha)=0\},
\end{equation}
where $f(\alpha)$ is $\lim_{x\searrow-1}f(x)$ or $\lim_{x\nearrow1}f(x)$ for
$\alpha=-1$ or $1$, respectively. It is obvious that $\widetilde F_\alpha$ is a closed
face of $\cF_\oc^+(-1,1)$ for every $\alpha\in[-1,1]$. For every $f\in\widetilde F_\alpha$
with $\alpha\in(-1,1)$, since $f(\alpha)=f'(\alpha)=0$, the expression of
Theorem \ref{T-2.1} gives
\begin{equation}\label{F-5.2}
f(x)=\int_{[-1,1]}{(x-\alpha)^2\over1-\lambda x}\,d\nu(\lambda),
\qquad x\in(-1,1),
\end{equation}
with a unique finite positive measure $\nu$ on $[-1,1]$ (defined by
$d\nu(\lambda):=(1-\lambda\alpha)^{-1}\,d\mu(\lambda)$ from $\mu$). This shows that
$\widetilde F_\alpha$ is a simplicial closed face of $\cF_\oc^+(-1,1)$ generated by the
extreme elements $(x-\alpha)^2/(1-\lambda x)$ ($-1\le\lambda\le1$). To prove the same
assertion when $\alpha=\pm1$, the following lemma is necessary. Although the integral
expression is in the same form as \eqref{F-5.2} with $\alpha=\pm1$, we need an independent
proof.

\begin{lemma}\label{L-5.1}
A real function $f$ on $(-1,1)$ belongs to $\widetilde F_{-1}$ if and only if there exists
a finite positive measure $\nu$ on $[-1,1]$ such that
$$
f(x)=\int_{[-1,1]}{(x+1)^2\over1-\lambda x}\,d\nu(\lambda),
\qquad x\in(-1,1).
$$
Furthermore, the measure $\nu$ is unique. A similar result holds for $\widetilde F_1$ with
$(x-1)^2$ in place of $(x+1)^2$ inside the integral.
\end{lemma}

\begin{proof}
The ``if\," part is easy. To prove the ``only if\,", assume that $f\in\widetilde F_{-1}$.
Since the function $h(x):=f(x)/(x+1)$ is a non-negative operator monotone function on
$(-1,1)$, the well-known integral representation (see e.g., \cite[Corollary V.4.5]{Bh})
says that there are a unique $a\in\bR$ and a unique finite positive measure $\mu$ on
$[-1,1]$ such that
\begin{equation}\label{F-5.3}
h(x)=a+\int_{[-1,1]}{x\over1-\lambda x}\,d\mu(\lambda),\qquad x\in(-1,1).
\end{equation}
Since
$$
0\le b:=\lim_{x\searrow-1}h(x)=a-\int_{[-1,1]}{1\over1+\lambda}\,d\mu(\lambda),
$$
we have $\int_{[-1,1]}(1+\lambda)^{-1}\,d\mu(\lambda)<+\infty$ (hence $\mu(\{-1\})=0$) and
$$
h(x)=b+\int_{(-1,1]}{x+1\over(1-\lambda x)(1+\lambda)}\,d\mu(\lambda)
=\int_{[-1,1]}{x+1\over1-\lambda x}\,d\nu(\lambda),\qquad x\in(-1,1),
$$
where $\nu$ is defined by $d\nu(\lambda):=(1+\lambda)^{-1}\,d\mu(\lambda)$ on $(-1,1]$ and
$\nu(\{-1\}):=b$. Therefore, $f$ admits the desired integral expression. The uniqueness of
$\nu$ is easily verified from that of $\mu$ in \eqref{F-5.3}.
\end{proof}

Now, we define
$$
\tilde g_{\alpha,\lambda}:={(x-\alpha)^2\over1-\lambda x}\qquad
(-1\le\alpha\le1,\ -1\le\lambda\le1).
$$
Continuing the above discussion, one can argue as in Section 3 to prove the following:

\begin{theorem}\label{T-5.2}
The extreme rays of $\cF_\oc^{++}(-1,1)$ are given by one of $\tilde g_{\alpha,\lambda}$
for $\alpha=\pm1$ and $-1\le\lambda\le1$. Also, the extreme rays of $\cF_\oc^+(-1,1)$ are
given by one of $\tilde g_{\alpha,\lambda}(x)$ for $-1\le\lambda\le1$ and
$-1\le\alpha\le1$.
\end{theorem}

For each $\alpha\in[-1,1]$ and each nonempty closed subset $\Lambda$ of $[-1,1]$, we
denote by $\widetilde F_{\alpha,\Lambda}$ the closed convex cone generated by
$\{\tilde g_{\alpha,\lambda}:\lambda\in\Lambda\}$, and by $\widetilde E_\Lambda$ the
closed convex cone generated by
$\{\tilde g_{\alpha,\lambda}:\alpha\in[-1,1],\,\lambda\in\Lambda\}$ and
$\{\tilde g_{-1,-1},\tilde g_{1,1}\}=\{x+1,1-x\}$. Then
$\widetilde F_\alpha=\widetilde F_{\alpha,[-1,1]}$ for every $\alpha\in[-1,1]$, where
$\widetilde F_\alpha$ is given in \eqref{F-5.1}. Moreover, for each
$f\in\cF_\oc^+(-1,1)$ we write $\Sigma_f$ ($\subset[-1,1]$) for the support of the
representing measure $\mu$ of $f$ in Theorem \ref{T-2.1}, and for each closed subset
$\Lambda$ of $[-1,1]$ define
$$
\widetilde\cF_\Lambda:=\{f\in\cF_\oc^+(-1,1):\Sigma_f\subset\Lambda\}.
$$
In particular,
$$
\widetilde\cF_\emptyset=\{p+qx:p\pm q\ge0\}.
$$

By arguments similar to those in Lemmas \ref{L-4.3}--\ref{L-4.6} we see that
$\widetilde\cF_\Lambda$ is a closed face of $\cF_\oc^+(-1,1)$ and
$$
\widetilde F_{\alpha,\Lambda}
=\widetilde F_\alpha\cap\widetilde\cF_\Lambda\quad(-1\le\alpha\le1),
\qquad\widetilde E_\Lambda=\widetilde\cF_\Lambda.
$$
Finally, one can obtain the following counterpart of Theorem \ref{T-4.1}. The proof is
more or less similar to that of Theorem \ref{T-4.1}, so the details are left to the reader.

\begin{theorem}\label{T-5.3}
The family of all closed faces ($\ne\{0\}$) of $\cF_\oc^+(-1,1)$ is given by
\begin{itemize}
\item[(i)] $\widetilde F_{\alpha,\Lambda}$ for $\alpha\in[-1,1]$ and for nonempty closed
subsets $\Lambda$ of $[-1,1]$,
\item[(ii)] $\widetilde E_\Lambda$ for closed subsets $\Lambda$ of $[-1,1]$.
\end{itemize}

Furthermore, the family of maximal closed faces of $\cF_\oc^+(-1,1)$ is given by
$\widetilde F_\alpha$ for $\alpha\in[-1,1]$.
\end{theorem}

\begin{corollary}\label{C-5.4}
The family of simplicial closed faces of $\cF_\oc^+(-1,1)$ is given by
$\widetilde F_{\alpha,\Lambda}$ for $\alpha\in[-1,1]$ and for nonempty closed subsets
$\Lambda$ of $[-1,1]$.
\end{corollary}

To show that $\widetilde E_\Lambda$ is not simplicial, one may note the identity
$$
g_{-1,\lambda}+g_{1,\lambda}+2g_{0,\lambda}
=4\bigl(g_{-1/2,\lambda}+g_{1/2,\lambda}\bigr).
$$

\bigskip
{\bf Acknowledgement.} The authors would like to thank Professor T.\ Ando who
proposed them the problem of determining simplicial subcones of the non-negative operator
convex functions on $(0,\infty)$. They are also grateful to a referee for constructive
comments that are helpful to improve the final version of the paper.
U.F.\ was supported by an ANR Project OSQPI (ANR-11-BS01-0008). F.H.\ also thanks the
Dept.\ of Mathematics of Universit\'e de Franche-Comt\'e where this work was mostly done
in the fall of 2010.

\bibliographystyle{amsplain}

\end{document}